\documentclass[11pt]{article}

\usepackage{amsmath, amsfonts, amssymb}
\usepackage{mathrsfs}
\usepackage{theorem}

\usepackage{bm}
\newtheorem{thm}{Theorem}[section]
\newtheorem{prop}[thm]{Proposition}

{\theorembodyfont{\rmfamily} \newtheorem{rem}[thm]{Remark}}
{\theorembodyfont{\rmfamily} }

\newcommand{\ra}{\rightarrow}

\newcommand{\R}{\mathbb R}

\newcommand{\N}{\mathbb N}

\newcommand{\rd}{\mathrm{d}}
\newcommand{\E}{\mathbb E}
\newcommand{\p}{\mathbb P}

\newcommand{\La}{\Lambda}
\newcommand{\veps}{\varepsilon}

\def\S{\mathcal S}
\def\diag{\mathrm{diag}}
\def\d{\mathrm{d}}

\newcommand{\fin}{\hspace*{\fill}\rule{0.3em}{1ex}}
\newenvironment{proof}{{\bf \noindent Proof.}}{\fin}

\numberwithin{equation}{section}

\begin{document}

\title{Ergodicity and first passage probability of regime-switching geometric Brownian motions\thanks{Supported in part by NSFC (No.11301030), NNSFC (No.11431014), 985-project.}}


\author{Jinghai Shao\footnote{Email:\ shaojh@bnu.edu.cn}\\[0.6cm] {School of Mathematical Sciences, Beijing Normal University, 100875, Beijing, China}}
\maketitle
\begin{abstract}
A regime-switching geometric Brownian motion is used to model a geometric Brownian motion with its coefficients changing randomly according to a Markov chain. In this work, we give a complete characterization of the recurrent property of this process. The long time behavior of this process such as its $p$-th moment is also studied. Moreover,  the quantitative properties of the regime-switching geometric Brownian motion with two-state switching are investigated to show the difference between geometric Brownian motion with switching and without switching. At last, some estimates of its first passage probability are established.
\end{abstract}

\section{Introduction}
In the study of mathematical finance, geometric Brownian motion (GBM) is used to model stock prices in the Black-Scholes model and is the most widely used model of stock price behavior. Let $(Z_t)$ be the solution of following stochastic differential equation (SDE):
\begin{equation}\label{GB}\rd Z_t=\mu Z_t\rd t+\sigma Z_t\rd B_t \end{equation}
with $Z_0=a>0$, where $\mu$, $\sigma$ are constants, and $(B_t)$ is a one-dimensional Brownian motion.  But GBM is not a completely realistic model, and there are several kinds of modifications of (\ref{GB}) to make it more realistic. For instance, the local volatility model and  the stochastic volatility model are well studied models in place of GBM (see \cite{FPS}).

In this work, we shall study another type of modification.
Consider the following regime-switching diffusion process $(X_t, \Lambda_t)_{t\geq 0}$, where
$X_t$ satisfies the SDE:
\begin{equation}\label{SGB}
\rd X_t =\mu_{\Lambda_t}X_t\rd t+\sigma_{\Lambda_t} X_t\rd B_t,
\end{equation}
with $X_0=x_0>0$, $\mu:\S\ra \R$, $\sigma:\S\ra (0,+\infty)$, $\S=\{0,\, 1,\ldots,N\}$. Here $(\Lambda_t)$ is a continuous time Markov chain on the finite state space $\S$ with $Q$-matrix $(q_{ij})_{i,j\in\S}$. Throughout this work, we assume that $(q_{ij})_{i,j\in\S}$ is irreducible and the process $(\La_t)$ is independent of the Brownian motion $(B_t)$.  We call $(X_t)_{t\geq 0}$ a regime-switching geometric Brownian motion (SGBM) in the state space $\S$.
The process $(X_t)$ defined by \eqref{SGB} can be viewed as a geometric Brownian motion living in a random environment which is characterized by a continuous time Markov chain $(\Lambda_t)$. When the process $(\La_t)$ takes different value in $\S$, it means that the environment is in different state. For example, if one uses the process \eqref{SGB} to model a stock price process living in a market which is divided into two kinds of different periods: ``bull" market and ``bear" market, one can take  $\S=\{0,1\}$ and use the state 0 to represent the ``bull" market, and the state 1 to represent the ``bear" market.
The stock market oscillates randomly between the bull market and the bear market, and it is conceivable that the drift $\mu$ in (\ref{GB}) in a bull market takes different value from that in a bear market.
Therefore, it is more practical to use (\ref{SGB}) instead of (\ref{GB}) to model stock price. Refer to \cite{Guo,Guo2} for more background on this model.  Refer to \cite{Guo1,GZ} and references therein for its application in option pricing.

The existence, uniqueness and non-explosiveness of $(X_t, \Lambda_t)$ are guaranteed by the general theory of regime-switching diffusion processes (cf. \cite{YZ} or \cite{Sh15c}). In particular, according to \cite[Lemma 7.1]{YZ}, similar to the geometric Brownian motion without switching, it still holds that
\begin{equation}\label{positive}
 \p\big(X_t\neq 0,\ \forall\,t\geq 0\big)=1,\quad \text{if}\ X_0=x_0>0.
\end{equation}
The process $(X_t)$ can also be expressed explicitly in the following form:
\begin{equation}\label{form-SGB}
X_t=x_0\exp\Big[\int_0^t\big(\mu_{\Lambda_s}-\frac 12\sigma_{\Lambda_s}^2\big)\rd s+\int_0^t \sigma_{\Lambda_s}\rd B_s\Big], \quad t>0.
\end{equation}

In this work, we shall first study the ergodic property and the long time behavior of the SGBM. Although the SGBM has been widely used in mathematical finance, its ergodic properties have not been well studied yet.
The SGBM is a simple example of regime-switching diffusion processes. It has been known that the ergodic properties of  regime-switching diffusion processes are more complicated than that of diffusion processes (see, for instance,  \cite{PS}).
Our present work presents quantitatively  how the coefficients of diffusion process and the switching rate of environment work together to impact the recurrent property of $(X_t,\Lambda_t)$.
Moreover, we refer the readers to \cite{B96, Gh, MY, PS}, \cite{Sh14}-\cite{SX14}, \cite{YX10} and references therein for the recent study on the recurrence, ergodicity, strong ergodicity, stability, and numerical approximation of regime-switching diffusion processes in a more general framework.

Let $(\pi_i)_{i\in \S}$ be the invariant probability measure of $(\La_t)$. Set
\begin{equation}
  \label{n-1}
  \begin{split}
  \Delta_i&=\mu_i-\frac12 \sigma_i^2,\ \lambda_i(p)= p\mu_i+\frac 12p(p-1)\sigma_i^2, \quad p>0,\ i\in \S,\\
  A_p&=Q+\diag(\lambda_0(p),\ldots,\lambda_N(p)),\\
  \eta_p&=-\max_{\gamma\in \mathrm{Spec}(A_p)}\mathrm{Re}\,\gamma,
  \end{split}
\end{equation}
where $\diag(\xi_0,\ldots,\xi_N)$ denotes  the diagonal matrix generated by the vector $(\xi_0,\ldots,\xi_N)$, and $\mathrm{Spec}(A_p)$ stands for the spectrum of $A_p$.
Our main results on the long time behavior and recurrence of SGBM are as follows.
\begin{thm}  \label{t-1}
\begin{itemize}
\item[$(\mathrm{i})$] If $\sum_{i\in\S} \pi_i\Delta_i>0$, then $\lim_{t\ra \infty} X_t=+\infty$ a.s. If $\sum_{i\in\S}\pi_i\Delta_i<0$, then $\lim_{t\ra \infty} X_t=0$ a.s.
\item[$(\mathrm{ii})$] For $p>0$, it holds
\begin{equation}
  \label{p-con} \lim_{t\ra\infty} \frac{\ln \E[X_t^p]}{t}=-\eta_p,
\end{equation} where $\eta_p$ is defined by \eqref{n-1}.
\end{itemize}
\end{thm}

\begin{thm}\label{t-2}
\begin{itemize}
  \item[$(\mathrm{i})$]  If $\sum_{i\in\S}\pi_i\Delta_i\neq 0$, then $(X_t,\La_t)$ is transient.
  \item[$(\mathrm{ii})$] If $\sum_{i\in\S}\pi_i\Delta_i=0$,
  then $(X_t,\La_t)$ is null recurrent.
\end{itemize}
\end{thm}

The argument of Theorem \ref{t-2} relies heavily on the Fredholm alternative. We  recall some basic facts on the Fredholm alternative. Recall that  $Q=(q_{ij})$ denotes the $Q$-matrix of the Markov chain $(\La_t)$. The equation $Qu=v$ is solvable if and only if $\sum_{i\in \S}\pi_i v_i=0$, in which case $u=Q^{-1}v$ is unique up to the addition of a multiple of the vector $\mathbf 1$. Moreover, $\sum_{i\in \S} \pi_i v_i (Q^{-1}v)(i)\leq 0$ for all $v$ satisfying $\sum_{i\in \S}\pi_i v_i=0$ and equality holds if and only if $v_\cdot\equiv 0$.

Secondly, we provide some quantitative description on the SGBM. The aim of this part is twofold: one is due to the requirement of the application of the SGBM; the other is that we want to find the complexity of the regime-switching diffusion processes via this simple linear model. In this part, we focus on the case that $\S=\{0,1\}$. Indeed, the formulas obtained in this part show that although the model of SGBM is simple, its quantitative properties are rather complicated. Here, we calculate the moments of $\ln X_t$ and estimate the first passage probability of $X_t$. We give out   the first and second order moments of $\ln X_t$, and every $n$-th order moment of $\ln X_t$ can be calculated by our method. Then we
provide an estimate of the first passage probability of the process $(X_t)$. The first passage probability plays important role in many research subjects such as in the option pricing and credit risk. Although the SGBM is rather simple from the point of view of stochastic differential equation, the calculation of its moments and its first passage probability is far from trivial. We need to overcome some new difficulties which do not occur in the study of the first passage probability of diffusion processes. Moreover, our results are analytic, and they are not expressed in terms of the Laplace transform. In \cite{Guo3,Hi} the distribution of the first passage time of SGBM in terms of its Laplace transform was studied; in \cite{KB} numerical approximation of the first passage probability for regime-switching processes was studied.

This work is organized as follows. In Section 2, we first study the long time behavior and recurrent properties of  $(X_t,\Lambda_t)$ in a finite state space.  The method depends on the criteria established in \cite{Sh15b} and \cite{PP}. The proofs of Theorem \ref{t-1} and Theorem \ref{t-2} are provided in this part. In Section 3, we calculate the moments of $\ln X_t$ when $(\La_t)$ is a Markov chain on a two-state space $\S=\{0,1\}$. The reason to focus on two-state space is that more explicit formula could be derived in this case.  There we provide explicit formula for the first and second moment of $\ln X_t$, and all its higher moments can be calculated using the same method. Then we provide some upper and lower bounds on the probability
$\p(\tau_a^{(x)}>T)$ when $\sigma_0=\sigma_1$, where $\tau_a^{(x)}=\inf\{t>0;\ X_t=a, X_0=x\}$ for $0<a<x=X_0$.  Via Slepian's lemma, we can transform the study of the first passage probability in the case $\sigma_0\neq \sigma_1$ into the case $\sigma_0=\sigma_1$, and this can provide us an upper estimate of $\p(\tau_a^{(x)}>T)$ when $\sigma_0\neq \sigma_1$.

\section{Long time behavior and recurrence of SGBM}
In this part, we shall study the long time behavior of  $(X_t, \Lambda_t)$ and provide a complete characterization of the recurrent property of $(X_t,\Lambda_t)$ by the method of Lyapunov functions.

Let us recall some basic definitions.
For $x\in (0,\infty) $, $i\in S$, define
\[\tau_{x,i}=\inf\{t\geq 0;\ (X_t, \Lambda_t)=(x,i)\}.\]
For every $x,\,y\in (0,\infty)$ and $i,\,j\in S$, if $\p_{x,i}(\tau_{y,j}<\infty)=1$, then the process $(X_t,\Lambda_t)$ is called \emph{recurrent}; if $\p_{x,i}(\tau_{y,j}<\infty)<1$, then it is called \emph{transient}; if $\E_{x,i}[\tau_{y,j}]<\infty$, then it is called \emph{positive recurrent}. If $(X_t,\La_t)$ is recurrent, but not positive recurrent, it is called \emph{null recurrent}. The process $(X_t,\Lambda_t)$ is called \emph{ergodic}, if there exists a probability measure $\tilde \pi$ on $(0,\infty)\times \S$ such that for every $x\in (0,\infty)$ and $i\in \S$
\[\|P_t((x,i), \cdot)-\tilde\pi\|_{\mathrm{var}}\ra 0,\quad \text{as $t\ra \infty$}.\]
Here $\|\cdot\|_{\mathrm{var}}$ denotes the total variation norm.
Moreover, if there exist  constants $\alpha,\,C(x,i)>0$ such that for every $(x,i)\in (0,\infty)\times \S$, and $t>0$,
\[\|P_t((x,i),\cdot)-\tilde\pi\|_{\mathrm{var}}\leq C(x,i)e^{-\alpha t}, \quad \text{as $t\ra \infty$},\]
then the process $(X_t,\Lambda_t)$ is called exponentially ergodic. These are usual definitions of recurrence for regime-switching processes, and we refer the reader to \cite[Chapter 3]{YZ} for more related discussion.

Let
\[L^{(i)} f(x)=\frac 12 \sigma_i^2x^2\frac{\rd^2 f}{\rd x^2}+\mu_ix\frac{\rd f}{\rd x}, \quad i\in \S, \ f\in C^2(\R).\]
For every function $g $ on $  \S$, define the operator
\[Q g(i)=\sum_{j\neq i}q_{ij}(g_j-g_i),\quad i\in\S.\]
Define
\begin{equation}\label{gen}
\mathscr Af(x,i)=L^{(i)} f(\cdot,i)(x)+Qf(x,\cdot)(i)
\end{equation}
for $f\in C^2(\R\times \S)$. Then $\mathscr A$ is the infinitesimal generator of $(X_t,\Lambda_t)$ (see \cite[Chapter II]{Sk}).

\noindent\textbf{Proof of Theorem \ref{t-1}}:\
(i)\ As $X_t>0$, $t\geq 0$ almost surely, we set $Y_t=\ln X_t$ and apply It\^o's formula to yield that
\begin{equation}
  \label{proc-y} \d Y_t=\d \ln X_t=\Delta_{\La_t} \d t +\sigma_{\La_t}\d B_t, \quad Y_0=\ln x_0,
\end{equation} where $\Delta_i$ is given by \eqref{n-1}.
By the ergodic theorem of Markov chains, we obtain that
\[\lim_{t\ra \infty}\frac{\ln X_t}{t}=\lim_{t\ra \infty} \frac 1t\int_0^t\Delta_{\La_s}\d s=\sum_{i\in \S}\pi_i\Delta_i\quad \mathrm{a.s.}\]
which yields immediately the assertion (i) of Theorem \ref{t-1}.

(ii)\ To make the idea clear, we provide a concise construction of the probability space. Let $(\Omega_1,\mathscr F^1,\p_1)$ be a probability space such that $(B_t)$ is a Brownian motion with respect to a given filtration $(\mathscr F_t)_{t\geq 0}$ on $(\Omega_1,\mathscr F^{1},\p_1)$. Let $(\Omega_2,\mathscr F^2,\p_2)$ be a probability space, and $(\La_t)$ be a Markov chain on it with the $Q$-matrix $(q_{ij})_{i,j\in  \S}$.
Define
\[ (\Omega,\mathscr F,\p)=(\Omega_1\times\Omega_2,\mathscr F^1\times \mathscr F^2,\p_1\times\p_2).\]
Then, in the following we let $(X_t)$ be a solution of \eqref{SGB} with respect to $(B_t)$ and $(\La_t)$ defined on  the probability measure $(\Omega,\mathscr F,\p)$.
Let $\E_{\p_1}$ denote taking expectation with respect to $\p_1$ and similarly define $\E=\E_{\p}=\E_{\p_1\times\p_2}$.

According to It\^o's formula, for $p>0$,
\[\d X_t^p=\big(p\mu_{\La_t}+\frac 12p(p-1)\sigma_{\La_t}^2\big)X_t^p\d t+p\sigma_{\La_t}X_t^p\d B_t.\]
For any $0\leq s<t$,
\begin{align*}
  X_t^p&=X_s^p+\int_s^t \big(p\mu_{\La_r}+\frac 12p(p-1)\sigma_{\La_r}^2\big)X_r^p\d r+\int_s^tp\sigma_{\La_r}X_r^p\d B_r.
\end{align*}
Hence,
\begin{equation}\label{ine-1}
\E_{\p_1}X_t^p=\E_{\p_1}X_s^p+\int_s^t\big(p\mu_{\La_r}+\frac 12p(p-1)\sigma_{\La_r}^2\big)\E_{\p_1} X_r^p\d r.
\end{equation}
Fix time $t>0$, and let $\tau_1<\tau_2<\ldots \tau_M$ be the jumping time of $(\La_t)$ during the period $(0,t)$. Set $\tau_{M+1}=t$, $\tau_0=0$.
Then $r\mapsto p\mu_{\La_r}+\frac 12p(p-1)\sigma_{\La_r}^2$ is continuous during $(\tau_k,\tau_{k+1})$ for $k=0,\ldots, M$. Hence, \eqref{ine-1} implies that
\[\frac{\d \E_{\p_1}X_r^p}{\d r}=\big(p\mu_{\La_r}+\frac 12p(p-1)\sigma_{\La_r}^2\big)\E_{\p_1} X_r^p,\quad r\in (\tau_{k},\tau_{k+1}),\]
and further
\begin{equation}
  \label{ine-2} \E_{\p_1}X_{\tau_{k+1}-}^p=\E_{\p_1}X_{\tau_{k}+}^p\exp\Big(\int_{\tau_{k}}^{\tau_{k+1}}
  p\mu_{\La_r}+\frac 12p(p-1)\sigma_{\La_r}^2\,\d r\Big),
\end{equation} for $k=0,\ldots, M$. Due to the continuity of $r\mapsto \E_{\p_1} X_r^p$,  we obtain that
\begin{align*}
  \E_{\p_1}X_t^p&=x_0^p\exp\Big(\int_0^t\big(p\mu_{\La_r}+\frac 12p(p-1)\sigma_{\La_r}^2\big)\d r\Big),
  \end{align*}
and hence
\[ \E X_t^p=x_0^p\,\E \exp\Big(\!\int_0^t\!\big(p\mu_{\La_r}\!+\!\frac 12p(p-1)\sigma_{\La_r}^2\big)\d r\Big)=x_0^p\,\E\exp\Big(\!\int_0^t\!\lambda(\La_r)\d r\Big).\]
According to \cite[Proposition 4.1]{BGM}, there exist constants $0<C_1(p)<C_2(p)<\infty$ such that
\[ C_1(p) e^{-\eta_pt}\leq \E \exp\Big(\int_0^t \lambda(\La_r)\d r\Big)\leq C_2(p) e^{-\eta_p t}.\]
Consequently,
\begin{gather*}
  \ln x_0^p+\ln C_1(p)-\eta_p t\leq \ln \E X_t^p\leq \ln x_0^p+ \ln C_2(p)-\eta_p t,\\
  \lim_{t\ra \infty}\frac{\ln \E X_t^p}{t}=-\eta_p,
\end{gather*}which is the desired conclusion.

Now we proceed to study the recurrent property of $(X_t,\La_t)$. Due to \eqref{positive}, $Y_t=\ln X_t$ is well-defined for $t\geq 0$ a.s. and satisfies the SDE \eqref{proc-y}.
It is obvious that the recurrent property  of $(X_t,\La_t)$ is equivalent to that of $(Y_t,\La_t)$. We shall use the Lyapunov method to justify the recurrent property of $(Y_t,\La_t)$.

\noindent\textbf{Proof of Theorem 1.2}:\
(i) We shall use the method of \cite[Section 3]{Sh15b}. Set $h(x)=|x|^{-2}$ and $g(x)=|x|^{-3}$ for $|x|\geq 1$. Define
\[\tilde L^{(i)} \rho(x)=\Delta_i \rho'(x)+\frac 12\sigma_i^2\rho''(x),\quad \rho\in C^{\infty}(\R), \ i\in \S.\]
Then
\[\tilde L^{(i)} h(x)=\Big(-2\Delta_i \frac{x}{|x|}+\frac{3\sigma_i^2}{|x|}\Big)g(x),\quad |x|>1.\]

Let us first consider the case $\sum_{i\in\S}\pi_i\Delta_i>0$.
As $\lim_{x\ra +\infty} -2\Delta_i\frac{x}{|x|}+\frac{3\sigma_i^2}{|x|}=-2\Delta_i$, there exist $\veps>0$, $r_1>1$ such that $\sum_{i\in\S} \pi_i(-2\Delta_i+\veps)<0$, and
\[-2\Delta_i\frac{x}{|x|}+\frac{3\sigma_i^2}{|x|}\leq -2\Delta_i+\veps\quad \forall\,x>r_1.\]
According to the Fredholm alternative, there exist a constant $\kappa>0$ and a vector $(\xi_i)_{i\in \S}$ so that
\[Q\xi(i)=-\kappa+2\Delta_i-\veps,\quad i\in\S.\]
Setting $f(x,i)=h(x)+\xi_i g(x)$ for $x>0$, we derive that
\begin{align*}
  \tilde{ \mathscr A} f(x,i)&=\tilde L^{(i)} h(x)+\xi_i\tilde L^{(i)}g(x)+Q\xi(i)g(x)\\
  &\leq \Big(-2\Delta_i+\veps+\xi_i\frac{\tilde L^{(i)} g(x)}{g(x)}+Q\xi(i)\Big)g(x)\\
  &=\Big(-\kappa+\xi_i\frac{\tilde L^{(i)}g(x)}{g(x)}\Big) g(x),\quad x>r_1,\, i\in\S.
\end{align*}
Here the operator $\tilde{\mathscr{A}}$ denotes the generator of the process $(Y_t,\La_t)$. As $\lim_{x\ra +\infty} \frac{\tilde{L}^{(i)} g(x)}{g(x)}=0$, there exists $r_2>r_1>1$ so that
\[-\kappa+\xi_i\frac{\tilde L^{(i)}g(x)}{g(x)} <0,\quad x>r_2,\,i\in \S,\]
and $x\mapsto h(x)+\xi_{\min} g(x)$ is a decreasing function on $[r_2,+\infty)$, where $\xi_{\min}=\min_{i\in\S} \xi_i$.

Take $Y_0=y>r_2$ and $\La_0=i_0$ so that $f(y,i_0)<h(r_2)+\xi_{\min} g(r_2)$. Set $\tau_K=\inf\{t>0;\ Y_t=K\}$, $K>r_2$, $\tau=\inf\{t>0;\ Y_t=r_2\}$.
By Dynkin's formula,
\[\E[f(Y_{t\wedge \tau_K\wedge\tau},\La_{t\wedge \tau_K\wedge\tau})]=f(y,i_0)+\E\int_0^{t\wedge \tau_K\wedge\tau}\tilde{\mathscr{A}}f(Y_s,\La_s)\d s\leq f(y,i_0).\]
Letting $t\ra +\infty$, we get
\[\E[f(K,\La_{\tau_K})\mathbf{1}_{\tau\geq \tau_K}]+\E[f(r_2,\La_{\tau})\mathbf{1}_{\tau,\tau_K}]\leq f(y,i_0),\]
which yields further that
\begin{equation}
  \label{ine-3} \p(\tau>\tau_K)\geq \frac{f(y,i_0)-h(r_2)-\xi_{\min}g(r_2)}{h(K)+\xi_{\min}g(K)-h(r_2)-\xi_{\min}g(r_2)}>0,
\end{equation} in the last step of which we have used the decreasing property of the function $h(x)+\xi_{\min}g(x)$ on $[r_2,+\infty)$. Invoking the fact $\tau_K\ra+\infty$ as $K\uparrow +\infty$ a.s., \eqref{ine-3} yields that
\[\p(\tau=+\infty)>0,\]
which implies that $(Y_t,\La_t)$ is transient.

Analogously, we can prove that $(Y_t,\La_t)$ is also transient when $\sum_{i\in\S}\pi_i\Delta_i<0$ by using the same Lyapunov function $f(x,i)$ but studying its behavior on $x<0$.

(ii)\ Now we consider the case $\sum_{i\in\S} \pi_i\Delta_i=0$. We shall apply the technique used in \cite{PP}. Set $\Theta=-\sum_{i\in \S} \pi_i\big(\Delta_i(Q^{-1}\Delta)(i)-\frac 12\sigma_i^2\big)$ and
\[f_0(x)=    \ln |x|.
\] Then, $\Theta>0$ due to the Fredholm alternative, and $\Theta f_0''(x)<0$ for $x\neq 0$ by direct calculation.
Define $f_1(x,i)$, $f_2(x,i)$ by
\begin{align*}
  f_1(x,i)&=-(Q^{-1} \Delta)(i) f_0'(x),\\
  f_2(x,i)&=Q^{-1}\big(\Delta_\cdot(Q^{-1}\Delta)(\cdot)-\frac 12\sigma_\cdot^2+\Theta\big)(i)f_0''(x).
\end{align*}
which  are well-defined due to the fact $\sum_{i\in\S}\pi_i\Delta_i=0$ and
\[\sum_{i\in \S}\pi_i\big( \Delta_i(Q^{-1} \Delta)(i)-\frac 12\sigma_i^2+\Theta\big)=0.\]
Take $f(x,i)=f_0(x)+f_1(x,i)+f_2(x,i)$ for $|x|\neq 0$, $i\in\S$.
Then
\begin{align*}
  \tilde{\mathscr{A}}f(x,i)&=\Delta_i f_0'(x)+\Delta_i f_1'(x,i)+\Delta_i f_2'(x,i)+\frac 12 \sigma_i^2 f_0''(x)\\
  &\quad +\frac1 2\sigma_i^2 f_1''(x,i)+\frac 12 \sigma_i^2 f_2''(x,i)+Qf_1(x,i)+Qf_2(x,i)\\
  &=-\Delta_i (Q^{-1} \Delta)(i) f_0''(x)+\frac 12 \sigma_i^2 f_0''(x)+\Delta_i f_2'(x,i)\\
  &\quad +\frac 12\sigma_i^2 f_2''(x,i)+Q f_2(x,i)\\
  &= \Theta f_0''(x)-\frac 12 \sigma_i^2(Q^{-1}\Delta)(i)f_0'''(x)+\Delta_i f_2'(x,i)+\frac 12 \sigma_i^2 f_2''(x,i).
\end{align*}
Since $f_0''(x)\ra 0$ as $|x|\ra +\infty$, it is easy to see that there exists $r_3>0$ such that
\[\tilde{\mathscr A} f(x,i)\leq 0\quad \forall\, |x|>r_3,\ i\in \S.\]
Moreover, $\lim_{|x|\ra \infty} f(x,i)=\lim_{|x|\ra \infty} f_0(x)+f_1(x,i)+f_2(x,i)=+\infty$, which yields that $(Y_t,\La_t)$,  and hence $(X_t,\La_t)$, is recurrent.

Next, we shall show further that when $\sum_{i\in\S}\pi_i\Delta_i=0$, $(X_t,\La_t)$, or equivalently $(Y_t,\La_t)$, is null recurrent.  To this aim, we take $1<q<p<2$ and set
\begin{gather*}
  h(x,i)=x^p+p\xi_i x^{p-1}, \ \ g(x,i)=x^q+q\xi_i x^{q-1}, \quad x>0,\ i\in\S,
\end{gather*}
where $\xi_i $ is determined by $Q\xi(i)=-\Delta_i$ for each $i\in\S$ satisfying further $\xi_i>0$, $i\in\S$.
Direct calculation yields that
\begin{align*}
  \tilde{\mathscr A} h(x,i)&=(\frac12 \sigma_i^2+\xi_i)p(p-1)x^{p-2}+\frac 12p(p-1)(p-2)\xi_i x^{p-3},\\
  \tilde{\mathscr A} g(x,i)&=(\frac 12 \sigma_i^2+\xi_i)q(q-1)x^{q-2}+\frac 12q(q-1)(q-2)\xi_i x^{q-3}
\end{align*} for $x>0,\ i\in\S$. As $p,\,q\in (1,2)$, it is easy to obtain that there exist constants $r_1,K>0$ such that
\[\tilde{\mathscr A} g(x,i)\geq 0,\quad 0\leq \tilde{\mathscr A} h(x,i)\leq K \ \text{for} \ x>r_1,\,i\in\S.\]
Let
\[u(x,i)=g(x,i)-C \ \ \text{for}\ x>0,\]
where $C>0$ satisfying $u(r_1,i)=g(r_1,i)-C<0$ for all $i\in \S$.
Set
\[R_n=\frac{\min_{i\in \S} h(n,i)}{\max_{i\in\S} u(n,i)},\]
then it is obvious that $\lim_{n\ra +\infty} R_n=+\infty$.
Denote $D=(0,r_1)$ and $\tau_D=\inf\{t>0; Y_t\in D\}$. Set $\tau_m=\inf\{t>0; Y_t\geq m\}$, then $\lim_{m\ra +\infty} \tau_m=+\infty$ a.s. Taking $Y_0=y>r_1$, $\La_0=i_0\in\S$ so that $u(y, i_0)>0$, Dynkin's formula yields that
\begin{align*}
  \E[(h\!-\!R_m u)(Y_{\tau_D\wedge\tau_m},\La_{\tau_D\wedge\tau_m})]\!-\!(h\!-\!R_m u)(y,i_0)&=\E\int_0^{\tau_D\wedge\tau_m}\!\!\!\tilde{\mathscr A}(h\!-\!R_m u)(Y_r,\La_r)\d r\\
  &\leq K \E[\tau_D\wedge\tau_m].
\end{align*}
Since $h(m,i)-R_m u(m,i)\geq 0$, and $h(r_1,i)-R_m u(r_1,i)\geq 0$ for every $i\in\S$ and $m$ large enough, we obtain
 \begin{equation}\label{ine-4}
 \E[\tau_D\wedge \tau_m]\geq \frac1K\big(R_m u(y,i_0)- h(y,i_0)\big).
 \end{equation}
Letting $m\ra +\infty$, \eqref{ine-4} yields that
\[\E[\tau_D]=+\infty.\]
Hence, $(Y_t,\La_t)$ is not positive recurrent. However, we have shown $(Y_t,\La_t)$ is recurrent. In all, $(Y_t,\La_t)$ is null recurrent.
The proof of Theorem \ref{t-2} is complete.

\section{Some quantitative  properties  of SGBM}
In this section, we shall exploit the quantitative properties of the process $(X_t,\La_t)$. Throughout this section, we only consider the case $(\La_t)$ is a Markov chain on $\S=\{0,1\}$. To emphasize this fact, in this section we denote the $Q$-matrix of $(\La_t)$ by
\[\begin{pmatrix}
-\lambda_0 &\lambda_0\\
\lambda_1 &-\lambda_1
\end{pmatrix},\]
where $\lambda_0,\,\lambda_1$ are two positive constants.  We calculate  the  moments of $\mathrm{ln}\,X_t$ in Subsection 3.1, and provide the estimate of the first passage probability of $(X_t)$ in Subsection 3.2.
\subsection{Moments of $\ln X_t$}\label{subs-mom}
Let us introduce some notations used in the sequel.
Assume that $\Lambda_0=i$ for $i=0$ or $1$. Set
\begin{align*}
&\zeta_1=\inf\{t>0;\ \Lambda_t= 1-i\},\quad
 \zeta_2=\inf\{t>\zeta_1;\ \Lambda_t= i\},\quad \cdots,\\
&\zeta_{2k-1}=\inf\{t>\zeta_{2k-2};\ \Lambda_t=1-i\},\quad
\zeta_{2k}=\inf\{t>\zeta_{2k-1};\ \Lambda_t=i\},\qquad k\geq 2.
\end{align*}
Let $\tau_k=\zeta_k-\zeta_{k-1}$ for $k\geq 1$ with $\zeta_0:=0$, then $(\tau_k)_{k\geq 1}$ are mutually independent random variables, and $\tau_{2k}$ and $\tau_{2k-1}$ are both exponentially distributed with parameters $\lambda_{1-i}$ and
$\lambda_i$ respectively.
Let $\alpha(t)$ and $\beta(t)$ denote respectively the time spent by  the process $(\Lambda_t)$ at the states $\Lambda_0$ and $1-\Lambda_0$ up to the time $t$.
Let $N(t)$ be the total number of transition between state $0$ and $1$ happened during $(0,t)$.

According to \cite{Ra}, the distribution of the sojourn time $\alpha(t)$ when $\Lambda_0=0$ is
\begin{equation}\label{alpha}
\begin{split}
\p(\alpha(t)\in\rd s)
&=\delta_t(s) e^{-\lambda_0 t}+\sum_{k=1}^\infty\Big[\frac{\lambda_0^{k}\lambda_1^{k-1}}{\Gamma(k)^2} s^{k-1}(t-s)^{k-1}e^{-\lambda_0 s}e^{-\lambda_1(t-s)}\\
&\quad +\frac{\lambda_0^{k}\lambda_1^{k}}{\Gamma(k)\Gamma(k+1)} s^k (t-s)^{k-1} e^{-\lambda_0 s}e^{-\lambda_1 (t-s)}\Big]\rd s, \quad s\in [0,t].
\end{split}
\end{equation}

Now we consider the geometric mean of the process $(X_t)$. Let
\begin{equation}
Y_t=\ln\Big(\frac{X_t}{X_0}\Big)=\int_0^t \Delta_{\Lambda_s}\rd s+\int_0^t\sigma_{\Lambda_s}\rd B_s.
\end{equation}
Recall that $\Delta_0$ and $\Delta_1$ are defined by \eqref{n-1}. Due to the independent increasing property of Brownian motion,
we can rewrite $Y_t$ in the form
\begin{equation}\label{form-y}
Y_t=\Delta_0 \alpha(t)+\Delta_1\beta(t)+\sigma_0\xi(\alpha(t))+\sigma_1\eta(\beta(t)),
\end{equation}
where $\xi(u)$ and $\eta(u)$ (for $u>0$) are mutually independent normally distributed random variables with mean $0$ and variance $u$, and $\xi(u),\,\eta(u)$ are independent of the process $(\Lambda_t)$.

\begin{prop}\label{mean}
Assume $\Lambda_0=0$. For $t>0$, it holds
\begin{equation}\label{e-mean-1}
\begin{split}\E[Y_t]&=(\Delta_0-\Delta_1)\E[\alpha(t)]+t\Delta_1\\
&=t\Delta_1+(\Delta_0\!-\!\Delta_1)e^{-\lambda_1 t}\sum_{k=1}^\infty\Big\{\frac{\lambda_0^k\lambda_1^{k-1}t^{2k}}{\Gamma(k)^2}\!\int_0^1u^{k}(1\!-\!u)^{k-1}e^{(\lambda_1\!-\!\lambda_0)tu}\rd u\\
&\quad+\frac{\lambda_0^k\lambda_1^k t^{2k+1}}{\Gamma(k)\Gamma(k+1)}  \! \int_0^1\!u^{k+1}(1\!-\!u)^{k-1}e^{(\lambda_1-\lambda_0)tu}\rd u\Big\}\!+\!(\Delta_0\!-\!\Delta_1)te^{-\lambda_0 t},
\end{split}
\end{equation}
and
\begin{equation}\label{e-mean-3}
\begin{split}
\E[Y_t^2]&= (\Delta_0^2t^2+\sigma_0^2t)e^{-\lambda_0 t}\\
&\quad +\sum_{k=1}^\infty \!\frac{\lambda_0^{k}\!\lambda_1^{k-1}t^{2k-1}}{\Gamma(k)^2}\!e^{\!-\lambda_1 t}\!\int_0^1\!\big(t^2(\Delta_0u\!+\!\Delta_1(1\!-\!u) )^2\!+\!\sigma_0^2 tu \!+\!\sigma_1^2t(1\!-\!u)\big)\\
&\qquad \qquad \qquad \qquad\qquad\qquad\cdot u^{k-1}(1\!-\!u)^{k-1}e^{(\lambda_1\!-\!\lambda_0)tu}\rd u\\
&\quad +\sum_{k=1}^\infty\! \frac{\lambda_0^k\lambda_1^{k}t^{2k}}{\Gamma(k)\Gamma(k+1)}\!e^{\!-\lambda_1 t}\!\int_0^1\!\big(t^2(\Delta_0u\!+\!\Delta_1(1\!-\!u) )^2\!+\!\sigma_0^2 tu \!+\!\sigma_1^2t(1\!-\!u)\big)\\
&\qquad\qquad\qquad\qquad\qquad\qquad \cdot u^k(1\!-\!u)^{k-\!1}e^{(\lambda_1\!-\!\lambda_0)tu}\rd u.
\end{split}
\end{equation}
In particular, when $\lambda_0=\lambda_1$,
\begin{align*}
\E[Y_t]&=t\Delta_1+(\Delta_0-\Delta_1)\frac{t e^{-\lambda_0t}}{2}\big[\cosh(\lambda_0t)+(1+\frac{1}{\lambda_0 t}) \sinh( \lambda_0 t )\big],\\
\E[Y_t^2]&=(\Delta_0^2t^2+\sigma_0^2t)e^{-\lambda_0 t}\\
&\quad+e^{-\lambda_0 t}\Big\{\frac1 4(\Delta_0^2+\Delta_1^2)t^2\Big(\frac{\cosh(\lambda_0 t)}{\lambda_0 t}+\sinh({\lambda_0 t})-\frac{\sinh({\lambda_0 t})}{(\lambda_0 t)^2}\Big)\\
&\quad+\frac12 (\sigma_0^2+\sigma_1^2) t\Big(-\frac{2}{\lambda_0 t}+\frac{\cosh(\lambda_0 t)}{(\lambda_0 t)^2}+\frac{\sinh(\lambda_0 t)}{(\lambda_0 t)^2}\Big)\\
&\quad + \frac 12 \Delta_0\Delta_1 t^2\Big(\cosh(\lambda_0 t)+ \sinh(\lambda_0 t)-\frac{\sinh(\lambda_0 t)}{\lambda_0 t}-\frac{\cosh(\lambda_0 t)}{\lambda_0 t}+\frac{\sinh(\lambda_0 t)}{(\lambda_0 t)^2}\Big)\\
&\quad+\frac 12 \sigma_1^2 t\Big(\cosh(\lambda_0 t)-\frac{\sinh(\lambda_0 t)}{\lambda_0 t}\Big)+\frac 1 4\Delta_1^2 t^2\Big(\cosh(\lambda_0 t)-\frac{\sinh(\lambda_0 t)}{\lambda_0 t}\Big) \\
&\quad+\frac 12 \sigma_0^2 t\Big(-2+\cosh(\lambda_0 t)+\frac{\sinh(\lambda_0 t)}{\lambda_0 t}\Big) +\frac 14 \Delta_0^2 t^2\Big(-4+\cosh(\lambda_0 t)+\frac{3\sinh(\lambda_0 t)}{\lambda_0 t}\Big)\Big\}.
\end{align*}
\end{prop}

\begin{proof}
By the distribution of    $\alpha(t),\, \xi(u),\eta(u)$, we get (\ref{e-mean-1})  by direct calculation.
By the formula (\ref{form-y}), due to the independence of $\xi(u)$ and $\eta(u)$, we get
\[\E[ Y_t^2]=\E\big[\Delta_0^2\alpha(t)^2+\Delta_1^2\beta(t)^2+2\Delta_0\Delta_1\alpha(t)\beta(t)+\sigma_0^2\alpha(t)+\sigma_1^2\beta(t)\big].\]
Then we obtain (\ref{e-mean-3}) according to the distribution of $\alpha(t)$. Using the Taylor expansion of functions $\cosh(x)$ and $\sinh(x)$ the formulae of $\E[Y_t]$ and $\E[Y_t^2]$ in the case $\lambda_0=\lambda_1$ can be obtained.
\end{proof}

Using the same method, every $n$-th order moments of $Y_t$ can be calculated.

\begin{prop}
Assume $\Lambda_0=0$. Then
\begin{equation}\label{limit-y-1}
\lim_{t\ra \infty} \frac{\E[Y_t]}{t}=\frac{\lambda_1\Delta_0+\lambda_0\Delta_1}{\lambda_0+\lambda_1},
\end{equation}
and
\begin{equation}\label{limit-y-2}
\lim_{t\ra \infty}\frac{\E[Y_t^2]}{t^2}=\frac{(\lambda_1\Delta_0+\lambda_0\Delta_1)^2}{(\lambda_0+\lambda_1)^2}.
\end{equation}
\end{prop}

\begin{proof}
According to the results \cite[Theorems 6, 7]{Tak}, it holds that
\begin{align*}
&\lim_{t\ra\infty} \frac{\E[\beta(t)]}{t}=\frac{\lambda_0}{\lambda_0+\lambda_1},\quad
\lim_{t\ra\infty}\frac{\mathrm{Var}[\beta(t)]}{t}=\frac{2\lambda_0\lambda_1}{(\lambda_0+\lambda_1)^3}.
\end{align*}
Hence, we have
\begin{align*}
&\lim_{t\ra \infty}\frac{\E[\alpha(t)]}{t}=\frac{\lambda_1}{\lambda_0+\lambda_1}, \ \ \lim_{t\ra \infty}\frac{\E[\alpha(t)^2]}{t^2}=\Big(\frac{\lambda_1}{\lambda_0+\lambda_1}\Big)^2.
\end{align*}
Then, by the formula (\ref{form-y}), we can get the desired results by direct calculation.
\end{proof}

\subsection{First passage probability of $(X_t)$}
For fixed $x>a>0$, let
\[\tau_a^{(x)}=\inf\{t>0;\ X_t=a, X_0=x\}.\]
What we are interested in is the  probability
\[\p(\tau_a^{(x)}\leq t)\quad \text{or }\ \p(\tau_a^{(x)}>t).\]
Although the first passage probability is very useful, the calculating of it is very difficult. Only for several simple cases, explicit formulas exist. For example, explicit formula exists for one-dimensional Brownian motion and piecewise monotone functionals of Brownian motion (cf. \cite{WP07}). To deal with general diffusion processes, one has to rely on some numerical approximation schemes. The regime-switching geometric Brownian motion provides us a simple example to see the difference between the study of first passage probability for diffusion  processes and the study of the first passage probability for regime-switching diffusion processes. Below, one can easily find that the switching of $(\La_t)$ causes new difficulty in calculating the first passage probability.

In this subsection, we consider first the case
$
\sigma_0=\sigma_1=\sigma>0$,
then the case $\sigma_0\neq \sigma_1$.
When $\sigma_0=\sigma_1=\sigma$, one gets by (\ref{form-SGB}) that
\begin{equation}\label{stop-1}\p(\tau_a^{(x)}>t)=\p\big(\min_{0\leq r\leq t}\big\{\sigma  B_r+\Delta_0\alpha(r)+\Delta_1\beta(r)+\tilde a\big\}>0\big),
\end{equation}
where $\tilde a=-\ln(a/x)>0$.
In order to provide upper and lower bounds of $\p(\tau_a^{(x)}>t)$, we introduce two auxiliary processes depending  on $(\Lambda_t)$. To simplify the notation,  we assume that
\begin{equation}\label{delta}
\Delta_0\leq \Delta_1.
\end{equation}
For fixed $T>0$, let
\begin{align*}
  g(t)&=\tilde a+\Delta_0\alpha(t)+\Delta_1\beta(t),\\
  g_u(t)&= \tilde a+\Delta_1 \min\{t,\beta(T)\}+ \Delta_0 (t-\beta(T))\mathbf{1}_{\beta(T)<t\leq T},\\
  g_l(t)&=\tilde a+\Delta_0 \min\{t,\alpha(T)\}+\Delta_1 (t-\alpha(T))\mathbf{1}_{\alpha(T)<t\leq T}.
\end{align*}
Then
\begin{equation}\label{bdy-b}
g_l(t)\leq g(t)\leq g_u(t),\quad  0\leq t\leq T.
\end{equation}
Indeed, by (\ref{delta}), we get, when $t\leq \beta(T)$,
\[g_u(t)-g(t)=\Delta_1 t-\Delta_0\alpha(t)-\Delta_1\beta(t)=(\Delta_1-\Delta_0)\alpha(t)\geq 0;\]
and, when $\beta(T)<t\leq T$,
\[g_u(t)-g(t)=(\Delta_1-\Delta_0)(\beta(T)-\beta(t))\geq 0.\]
Similarly, we can prove the first inequality of (\ref{bdy-b}).

Set $\Phi(z)=\int_{-\infty}^z \frac 1{\sqrt{2\pi}} e^{-\frac{y^2}{2}}\rd y$,
\begin{align*}
F_u(t)&=\frac{1}{\sqrt{2\pi(T-t)}}\int_{-\infty}^0\Big(1-e^{\frac{2\tilde a y}{\sigma(T-t)}}\Big)\Big[\Phi\Big(\frac{\Delta_0 t/\sigma-y}{\sqrt{t}}\Big) \\
&\quad -e^{2\Delta_0 y/\sigma}\Phi\Big(\frac{\Delta_0 t/\sigma +y}{\sqrt{t}}\Big)\Big] e^{-\frac{(y+\Delta_1(T-t)/\sigma+\tilde a/\sigma)^2}{2(T-t)}}\,\rd y,\quad 0<t< T,
\end{align*}
and
\begin{align*}
F_l(t)&=\frac{1}{\sqrt{2\pi t}}\int_{-\infty}^0\Big(1-e^{\frac{2\tilde a y}{\sigma t}}\Big)\Big[\Phi\Big(\frac{\Delta_1 (T-t)/\sigma-y}{\sqrt{T-t}}\Big) \\
&\quad -e^{2\Delta_1 y/\tilde a}\Phi\Big(\frac{\Delta_1 (T-t)/\sigma +y}{\sqrt{T-t}}\Big)\Big] e^{-\frac{(y+\Delta_0 t/\sigma+\tilde a/\sigma)^2}{2t}}\,\rd y, \quad 0<t<T.
\end{align*}

\begin{thm}\label{t-fpp-1}
Assume $\sigma_0=\sigma_1=\sigma>0$, $\Delta_1\geq \Delta_0$. Suppose that the process $(\Lambda_t)$ starts from 0, i.e. $\Lambda_0=0$ a.s. For every $T>0$,
it holds
\begin{equation}\label{fpp-u}
\begin{split}
\p(\tau_a^{(x)}>T)&\leq \Big(\Phi\Big(\frac{\Delta_0 T +\tilde a}{\sigma \sqrt{T}}\Big)-e^{-2\Delta_0\tilde a/\sigma^2} \Phi\Big(\frac{\Delta_0 T-\tilde a}{\sigma\sqrt{T}}\Big)\Big)e^{-\lambda_0 T}\\
&\quad +\sum_{k=1}^\infty \frac{\lambda_0^{k-1}\lambda_1^{k-1}}{\Gamma(k)^2} \int_0^T F_u(t) t^{k-1}(T-t)^{k-1} e^{-\lambda_0 t}e^{-\lambda_1(T-t)}\rd t\\
&\quad+\sum_{k=1}^\infty \frac{\lambda_0^k\lambda_1^k}{\Gamma(k)\Gamma(k+1)} \int_0^T F_u(t) t^{k}(T-t)^{k-1} e^{-\lambda_0 t}e^{-\lambda_1 (T-t)}\rd t,
\end{split}
\end{equation}
and
\begin{equation}\label{fpp-l}
\begin{split}
\p(\tau_a^{(x)}>T)&\geq \Big(\Phi\Big(\frac{\Delta_0 T +\tilde a}{\sigma \sqrt{T}}\Big)-e^{-2\Delta_0\tilde a/\sigma^2} \Phi\Big(\frac{\Delta_0 T-\tilde a}{\sigma\sqrt{T}}\Big)\Big)e^{-\lambda_0 T}\\
&\quad +\sum_{k=1}^\infty \frac{\lambda_0^{k-1}\lambda_1^{k-1}}{\Gamma(k)^2} \int_0^T F_l(t) t^{k-1}(T-t)^{k-1} e^{-\lambda_0 t}e^{-\lambda_1(T-t)}\rd t\\
&\quad+\sum_{k=1}^\infty \frac{\lambda_0^k\lambda_1^k}{\Gamma(k)\Gamma(k+1)} \int_0^T F_l(t) t^{k}(T-t)^{k-1} e^{-\lambda_0 t}e^{-\lambda_1 (T-t)}\rd t.
\end{split}
\end{equation}
\end{thm}

\begin{proof}
By (\ref{stop-1}) and (\ref{bdy-b}), we get
\[\p(\tau_a^{(x)}>T)\leq \p(\min_{0\leq t\leq T}\{\sigma B_t+g_u(t)\}>0)=\p(\max_{0\leq t\leq T}\{\sigma B_t-g_u(t)\}<0).
\]
Due to the independence of $(\Lambda_t)$ and $(B_t)$,
\begin{align*}
&\p(\max_{0\leq t\leq T}\{\sigma B_t-g_u(t)\}<0)=\E\big[\E\big[\mathbf{1}_{\max\limits_{0\leq t\leq T} \{\sigma B_t-g_u(t)\}<0}\big|\mathscr F_{\Lambda} \big]\big]\\
&=\E\big[\E\big[\mathbf{1}_{\max\limits_{0\leq t\leq \beta(T)}\{\sigma B_t-\tilde a-\Delta_1 t\}<0, \max\limits_{\beta(T)<t\leq T}\{\sigma B_t-\tilde a-\Delta_1\beta(T)-\Delta_0(t-\beta(T))\}<0}\big|\mathscr F_\Lambda\big]\big]\\
&=\E\big[\E\big[\mathbf{1}_{\max\limits_{0\leq t\leq \beta(T)}\{\sigma B_t-\tilde a-\Delta_1 t\}<0, \max\limits_{0<\!t\leq \alpha(T)}\{\sigma B_t-\tilde a-\Delta_1 \beta(T)+\sigma B(\beta(T))\!-\!\Delta_0 t\}<0}\big|\mathscr F_\Lambda\big]\big],
\end{align*}
where $\mathscr F_\Lambda =\sigma(\Lambda_t;\ t\leq T)$.
According to the well known results on the first passage probability of Brownian motion (cf. for example, \cite[pp 375]{Si}),
\[\p\big(\max_{0\leq t\leq \beta(T)}\{ \sigma B_t-\tilde a-\Delta_1 t\}<0\big|B_{\beta(T)}=y\big)=1-\exp\Big(-2\frac{\tilde a(\Delta_1 \beta(T)+ \tilde a -\sigma y)}{\sigma^2\beta(T)} \Big),\]
and
\begin{align*}
&\p\big(\max_{0\leq t\leq \alpha(T)}\{\sigma B_t+\sigma y-\tilde a-\Delta_1 \beta(T)-\Delta_0 t\}<0\big)\\
&=\Phi\Big(\frac{\Delta_0\alpha(T)+\Delta_1\beta(T)+\tilde a-\sigma y}{\sigma \sqrt{\alpha(T)}}\Big)-e^{-2\frac{\Delta_0(\tilde a+\Delta_1 \beta(T)-\sigma y)}{\sigma ^2}}\Phi\Big(\frac{\Delta_0\alpha(T)-\Delta_1\beta(T)-\tilde a +\sigma y}{\sigma \sqrt{\alpha(T)}}\Big).
\end{align*}
Consequently,
if $\alpha(T)=T$, then
\begin{align*}
&\p(\max_{0\leq t\leq T}\{\sigma B_t-g_u(t)\}<0|\alpha(T)=T)=\p(\max_{0\leq t\leq T}\{\sigma B_t-\Delta_0 t-\tilde a\}<0)\\
&=\Phi\Big(\frac{\Delta_0 T+\tilde a}{\sigma \sqrt{T}}\Big)-e^{-2\frac{\Delta_0 \tilde a}{\sigma^2}}\Phi\Big(\frac{\Delta_0T-\tilde a}{\sigma \sqrt{T}}\Big).
\end{align*}
If $0\!<\!\alpha(T)\!<\!T$, we have
\begin{align*}
&\p\big(\max_{0\leq t\leq T}\{\sigma B_t-g_u(t)\}<0\big|0\!<\!\alpha(T)\!<\!T\big)&\\
&\leq  \E\Big[\int_{-\infty}^{(\Delta_1\beta(T)+\tilde a)/\sigma}\Big(1-\exp\Big(-2\frac{\tilde a(\Delta_1 \beta(T)+ \tilde a -\sigma y)}{\sigma^2\beta(T)} \Big)\Big)\\
&\quad\cdot\Big(\Phi\Big(\frac{\Delta_0\alpha(T)\!+\!\Delta_1\beta(T)\!+\!\tilde a\!-\!\sigma y}{\sigma \sqrt{\alpha(T)}}\Big)\!-\!e^{-2\frac{\Delta_0(\tilde a\!+\!\Delta_1 \beta(T)\!-\!\sigma y)}{\sigma ^2}}\Phi\Big(\frac{\Delta_0\alpha(T)\!-\!\Delta_1\beta(T)\!-\!\tilde a \!+\!\sigma y}{\sigma \sqrt{\alpha(T)}}\Big)
 \Big)\\
 &\quad\cdot \frac{1}{\sqrt{2\pi \beta(T)}} e^{-\frac{y^2}{2\beta(T)}}\rd y \Big|0\!<\!\alpha(T)\!<\!T\Big]\Big]\\
 &=\E\Big[\E\Big[ \int_{-\infty}^0\Big(1-e^{\frac{2\tilde a z}{\sigma \beta(T)}}\Big)\Big(\Phi(\frac{\Delta_0\alpha(T)-\sigma z}{\sigma\sqrt{\alpha(T)}}\Big)
-e^{2\frac{\Delta_0  z}{\sigma}}\Phi\Big(\frac{\Delta_0\alpha(T)+\sigma z}{\sigma\sqrt{\alpha(T)}}\Big)\Big)\\
 &\quad \cdot \frac{1}{\sqrt{2\pi \beta(T)}} e^{-\frac{(\sigma z+\Delta_1\beta(T)+\tilde a)^2}{2\sigma^2\beta(T)}}\rd z  \Big|0\!<\!\alpha(T)\!<\!T\Big]\Big].
\end{align*}
Under the assumption $\Lambda_0=0$, one has $\p(\alpha(T)>0)=1$.
Then, invoking the distribution of $\alpha(T)$ and the definition of $F_u(t)$, we can get (\ref{fpp-u}). By a similar argument, we can  get (\ref{fpp-l}), and the proof is completed.
\end{proof}

Now we consider the first passage probability of $(X_t)$ when $\sigma_0\neq \sigma_1$.  Note that $\sigma_0,\,\sigma_1$ stand for the volatility, so it is not restrictable to assume that $\sigma_0,\,\sigma_1>0$.

\begin{prop}
Assume that $\sigma_0>\sigma_1>0$. For each $0<a<x=X_0$ and $T>0$, it holds
\begin{equation}
\begin{split}
&\p(\tau_a^{(x)}>T\big|\mathscr F_\Lambda)\\
&=\p\big(\min_{0\leq t\leq T}\big\{\int_0^t\sigma_{\Lambda_s}\rd B_s+\int_0^t\Delta_{\Lambda_s}\rd s-\ln(a/x)\big\}>0\big|\mathscr F_\Lambda\big)\\
&\leq\p\big(\min_{0\leq t\leq T}\big\{\sqrt{\sigma_0^2-\sigma_1^2}B_{\alpha(t)}+\sqrt{t}\sigma_1\eta_0+\Delta_0\alpha(t)+\Delta_1\beta(t)-\ln(a/x)\big\}>0\big|\mathscr F_{\Lambda}\big),
\end{split}
\end{equation}
where $\eta_0$ is a standard normally distributed random variable, and is independent of $(B_t, \Lambda_t)$.
\end{prop}

\begin{proof}
For clarity of the idea, let us introduce a probability space $(\Omega,\p)$. Let $\Omega=\Omega_1\times\Omega_2=C([0,\infty);\ \R\times S)$. Then there exists  a probability measure $\p=\p_1\times \p_2$ on $\Omega$ such that $\omega=(\omega_1(\cdot),\omega_2(\cdot))$ satisfying that $(\omega_1(t))$ is a Brownian motion under $\p_1$ on $\Omega_1$ and $(\omega_2(t))$ is a $Q$-process with $Q$-matrix $\begin{pmatrix} -\lambda_0&\lambda_0\\ \lambda_1&-\lambda_1\end{pmatrix}$ under $\p_2$ on $\Omega_2$. Set $B_t(\omega)=\omega_1(t)$ and $\Lambda_t(\omega)=\omega_2(t)$, then under $\p$, $(B_t)$ and $(\Lambda_t)$ satisfy the condition used in the definition of the process $(X_t, \Lambda_t)$.
Moreover,
\begin{align*}
&\p\big(\min_{t\leq T}\big\{\int_0^t\sigma_{\Lambda_s}\rd B_s+\Delta_0\alpha(t)+\Delta_1\beta(t)-\ln(a/x)\big\}>0\big|\sigma(\Lambda_t;t\leq T)\big)\\
&=\p_1\big(\max_{t\leq T}\big\{\int_0^t\sigma_{\Lambda_s}\rd B_s-\Delta_0\alpha(t)+\Delta_1\beta(t)+\ln(a/x)\big\}<0\big).
\end{align*}
Recall the definition of $\zeta_k$ in the beginning of Subsection \ref{subs-mom}. For $t\in (\zeta_{2k},\zeta_{2k+1}]$, we have
\begin{align*}
&\int_0^t\sigma_{\Lambda_s}\rd B_s=\sigma_0 B_{\zeta_1}+\sigma_1 (B_{\zeta_2}-B_{\zeta_1})+\ldots+\sigma_0(B_t-B_{\zeta_{2k}})\\
&=\sigma_0B^{(1)}(\alpha(t))+\sigma_1B^{(2)}(\beta(t)),
\end{align*}
where $B^{(1)}(t)$ and $B^{(2)}(t)$ are independent normally distributed random variables for each $t>0$, which satisfies
\[\E_{\p_1}[B^{(i)}(t)]=0,\ \E_{\p_1}[B^{(i)}(t)^2]=t,\quad i=1,\,2.
\]
For $N\in \N$, and $k\in \N$ such that $k/N\leq T$,
set
\begin{equation}
Y_k=\int_0^{k/N}\sigma_{\Lambda_s}\rd B_s=\sigma_0B^{(1)}(\alpha(k/N))+\sigma_1B^{(2)}(k/N).
\end{equation}
Then $(Y_k)_k$ are Gaussian random variables satisfying
\[\E_{\p_1} [Y_k]=0,\ \E_{\p_1}[Y_k^2]=\sigma_0^2\alpha(k/N)+\sigma_1^2\beta(k/N).\]
Moreover, for every $k<j$ with $j/N\leq T$,
\[Y_j=\int_0^{j/N}\sigma_{\Lambda_s}\rd B_s=Y_k+\int_{k/N}^{j/N}\sigma_{\Lambda_s}\rd B_s.\]
Hence, \[
\E_{\p_1}[Y_kY_j]=\E_{\p_1}[Y_k^2]=\sigma_0^2\alpha(k/N)+\sigma_1^2\beta(k/N).\]

Let $(\widetilde B_t)$ be a Brownian motion, and $\eta_0$ be a standard normally distributed random variable under $\p_1$ on $\Omega_1$, which is independent of $(\widetilde B_t)$.
Set
\begin{equation}
Z_k=\sqrt{\sigma_0^2-\sigma_1^2} \widetilde B_{\alpha(k/N)}+\sqrt{k/N} \sigma_1\eta_0,\quad k/N\leq T.
\end{equation}
then $(Z_k)_k$ are Gaussian random variables with
\[\E_{\p_1}[Z_k]=0,\quad \E_{\p_1}[Z_k^2]=\sigma_0^2\alpha(k/N)+\sigma_1^2\beta(k/N)=\E_{\p_1}[Y_k^2],\]
and for $k<j$ with $j/N\leq T$,
\begin{align*}
&\E_{\p_1}[Z_k Z_j]=(\sigma_0^2-\sigma_1^2)\alpha(k/N)+\sigma_1^2\sqrt{\frac{kj}{N^2} }>(\sigma_0^2-\sigma_1^2)\alpha(k/N)+\sigma_1^2k/N=\E_{\p_1}[Y_kY_j].
\end{align*}
Therefore, according to Slepian's lemma (cf. \cite[pp 74]{LT}), we have
\begin{align*}
&\p_1\big(\max_{k;k/N\leq T}\big\{Y_k-\Delta_0\alpha(k/N)-\Delta_1\beta(k/N)+\ln(a/x)\big\}<0\big)\\
&\leq \p_1\big(\max_{k;k/N\leq T}\big\{Z_k-\Delta_0\alpha(k/N)-\Delta_1\beta(k/N)+\ln(a/x)\big\}<0\big)\\
&=\p_1\big(\max_{k;k/N\leq T}\big\{\sqrt{\sigma_0^2-\sigma_1^2}\widetilde B_{\alpha(k/N)}+\sqrt{k/N}\sigma_1\eta_0-\Delta_0\alpha(k/N)-\Delta_1\beta(k/N)+\ln(a/x)\big\}<0\big).
\end{align*}
Letting $N\ra \infty$, we get
\begin{align*}
&\p_1\big(\max_{t\leq T}\big\{\int_0^t \sigma_{\Lambda_s}\rd B_s-\Delta_0\alpha(t)-\Delta_1\beta(t)+\ln(a/x)\big\}<0\big)\\
&\leq \p_1\big(\max_{t\leq T}\big\{\sqrt{\sigma_0^2\!-\!\sigma_1^2}\widetilde B_{\alpha(t)}+\sqrt{t}\sigma_1\eta_0-\Delta_0\alpha(t)-\Delta_1\beta(t)+\ln(a/x)\big\}<0\big).
\end{align*}
Then applying the symmetry of the Brownian motion, we get the desired conclusion.
\end{proof}
\begin{rem}
According to this proposition, we can transform the situation $\sigma_0\neq \sigma_1$ to the situation $\sigma_0=\sigma_1$ via Slepian's lemma. Combining with Theorem \ref{t-fpp-1}, we can obtain an upper bound of $\p(\tau_a^{(x)}>T)$ when $\sigma_0\neq \sigma_1$.
\end{rem}

\noindent\textbf{Acknowledgments.}\ The author is grateful to professors Mu-Fa Chen and Yong-Hua Mao for their valuable discussion.

\end{document}